\documentclass[11pt,reqno]{amsart}
\usepackage{setspace}
\usepackage{amsmath}
\usepackage{amscd,amsthm,amsfonts,amsopn,amssymb,verbatim}
\usepackage{mathrsfs}
\numberwithin{equation}{section}

\newtheorem{theorem}{Theorem} 

\newtheorem{proposition}[theorem]{Proposition}

\newtheorem{lemma}{Lemma}

\theoremstyle{remark}

\def\be{\begin{equation}}
\def\ee{\end{equation}}

\def\ve{\varepsilon}

\def\Prob{\rm Prob\,}
\def\Spec{\,Spec\,}
\def\dist{\rm dist}
\allowdisplaybreaks
\begin{document}

\title[]
{On the local eigenvalue spacings for certain Anderson-Bernoulli Hamiltonians}
\author
{J.~Bourgain}
\address
{Institute for Advanced Study, Princeton, NJ 08540}
\email
{bourgain@math.ias.edu}
\thanks{This work was partially supported by NSF grant DMS-1301619}
\date{\today}

\abstract
The aim of this work is to extend the results from [B2] 
on local eigenvalue spacings to certain 1D lattice Schrodinger
with a Bernoulli potential.  
We assume the disorder satisfies a certain algebraic condition that enables one to invoke the recent
results from [B1] on the regularity of the density of states.
In particular we establish Poisson local eigenvalue statistics in those models.
\endabstract

\maketitle

\section{\bf Introduction}

The aim of this Note is to exploit the results from \cite{B1} on certain
Anderson-Bernoulli (A-B) Hamiltonians, in order to extend some of the eigenvalue
spacing properties obtained in \cite{B2} for Hamiltonians with H\"older
site-distribution to the A-B setting.

As in \cite{B2}, all models are 1D.
Recall that the A-B Hamiltonian with coupling $\lambda$ is given by
\be\label {1.1}
H=H_\lambda =\Delta +\lambda V
\ee
where $V=(v_1)_{n\in\mathbb Z} $ are IID-variables ranging in $\{-1, 1\}$,
$\Prob [v_n= -1] =\frac 12=\Prob [v_n=1]$.
It is believed that for $\lambda\not=0$ sufficiently small, the integrated density
of states (IDS) $\mathcal N$ of $H$ is Lipschitz and becomes arbitrary smooth for
$\lambda\to 0$.
A first result in this direction was obtained in \cite {B1}, for small $\lambda$
with certain specific algebraic properties.

\begin{proposition}\label{Proposition1}
(see \cite{B1}).

Let $H_\lambda$ be the A-B model considered above and restrict $|E|< 2-\delta_0$ for some fixed $\delta_0>0$.
Given a constant $C>0$ and $k\in \mathbb Z_+$, there is some $\lambda_0=\lambda_0(C, k)> 0$ such that $\mathcal N(E)$
is $C^k$-smooth on $[-2+\delta_0, 2-\delta_0]$ provided $\lambda$ satisfies the following conditions.

\begin{itemize}
\item[(1.2)] \ $|\lambda|<\lambda_0$
\item[(1.3)] \ $\lambda$ is an algebraic number of degree $d<C$ and minimal polynomial $P_d(x)\in\mathbb Z[X]$ with 
coefficients bounded by $(\frac 1\lambda)^C$
\item[(1.4)] \ $\lambda$ has a conjugate $\lambda'$ of modulus $|\lambda'|\geq 1$
\end{itemize}
\end{proposition}

In what follows, we assume $\lambda$ satisfies the conditions of Proposition \ref{Proposition1} and the energy $E$ restricted
to $[-2+\delta_0, 2-\delta_0]$, unless specified differently.

Once we are in the presence of the Hamiltonian with a bounded density of states $k(E)=\frac{ d\mathcal N}{dE}$, it becomes a
natural question to inquire about local eigenvalue statistics for `truncated' models $H_N$, denoting $H_N$ the restriction of
$H$ to the interval $[1, N]$ with Dirichlet boundary conditions.
This problem was explored in \cite{B2}, assuming the site distribution $v_n$ of $V$ H\"older regular of some exponent
$\beta>0$, and we extended (in 1D) the theorem from \cite{G-K} on Poisson statistics in this setting.
Here we consider the A-B situation.

\begin{proposition}\label{Proposition2}
With $H$ as in Proposition 1, the rescaled eigenvalues of $H_N$
$$
\{N(E-E_0)\mathcal X_I(E)\}_{E\in\Spec H_N}
$$
where $I=[E_0, E_0+\frac LN]$ and we let first $N\to\infty$, then $L\to\infty$, obey Poisson statistics.
\end{proposition}

This is the analogue of \cite {B2}, Proposition 5.
Again one could conjecture the above statement to hold under the sole assumption that $\lambda$ be sufficiently small.

Proposition 2 gives a natural example of a Jacobi Schr\"odinger operator with bounded density of states where the local
eigenvalue spacing distribution differs from that of the potential (cf. S.~Jitomirskaya's talk `Eigenvalue statistics for ergodic 
localization', Berkeley 11/10/2010.

Even with the smoothness of the IDS at hand, the arguments from \cite{B2} do not carry out immediately to the A-B setting.
For instance, the `classical' approach to Minami's inequality (see \cite{C-G-K}) rests also on regularity of the single site
distribution (in addition to a Wegner estimate) which makes it inapplicable in the A-B case.
This will require us to develop an alternative argument in order 
to deal with near resonant eigenvalues.

Roughly speaking, it turns out that for the analysis below, the following ingredients suffice.

\begin{itemize}
\item[(1)] The Furstenberg measures $\nu_E$ are absolutely continuous with bounded density
\item[(2)] The density of states $k$ is $C^1$
\end{itemize}

Hence the results from \cite{B2} for H\"older site distribution follow from the present treatment.
We believe however that the presentation in \cite {B2} remains of interest since it is considerably simpler than
the method from this paper.

As in \cite{B2}, the techniques are very much 1D and based on the usual
transfer matrix formalism.

Recall that
$$
M_n=M_n(E) =\prod^1_{j=n} \begin{pmatrix} E-v_j & -1\\ 1&0\end{pmatrix}
\eqno{(1.5)}
$$
and that the equation $H\xi=E\xi$ on the positive side is equivalent to
$$
M_n\begin{pmatrix} \xi_1\\ \xi_0\end{pmatrix} =\begin{pmatrix}
\xi_{n+1}\\ \xi_n\end{pmatrix}.\eqno{(1.6)}
$$
What follows will use extensively ideas and techniques developed in 
\cite{B1}, \cite{B2}.
\medskip

\section{Preliminary estimates}

Denote $\nu_E$ the Furstenberg measure at energy $E$.
This is the unique probability measure on $P_1(\mathbb R)\simeq S^1$
which is $\mu =\frac 12 (\delta_{g_+}+\delta_{g_-})$ - stationary, where
$$
g_+ = \begin{pmatrix} E+\lambda&-1\\ 1&0\end{pmatrix} \text { and }
g_- =\begin{pmatrix} E-\lambda& -1\\ 1&0\end{pmatrix}.\eqno{(2,1)}
$$
Thus
$$
\nu_E=\sum_g\mu(g) \tau_g^*[\nu_E]\eqno{(2.2)}
$$
where $\tau_g$ denotes the projective action of $g\in SL_2(\mathbb R)$.

It was proven in \cite{B1} that in the context of Proposition 1, $\nu_E$
is absolutely continuous wrt Haar measure on $S^1$ and moreover $\frac
{d\nu_E}{d\theta}$ becomes arbitrarily smooth for $\lambda \to\infty$.

The results of this section are stated for A-B Hamiltonians in general
however and rely on general random matrix product theory.

\begin{lemma}\label{Lemma1}
Let $\xi, \eta$ be unit vectors in $\mathbb R^2$.
Given $E$ and $\ve>0$, $N>C(\lambda)\log\frac 1\ve$ we have
$$
\mathbb E[|\langle M_N(E)(\xi), \eta\rangle| < \ve \Vert
M_N(E)(\xi)\Vert]\leq\tau(\ve)\eqno{(2.3)}
$$
where $\tau(\ve)=\max_{|I|=\ve} \nu_E(I), I\subset S^1$ an interval.
\end{lemma}

\begin{proof}

Let $\xi =e^{i\theta}$, $\eta= e^{i\psi}$.
If $M_N= g_N\cdots g_1, g_j\in \{g_+, g_-\}$, then
$$
\Big|\frac{\langle M_N(\xi), \eta\rangle}{\Vert M_N(\xi)\Vert}\Big|=
\big|\cos (\tau_{g_N\cdots g_1} (\theta) -\psi)\big|.
$$
Hence the l.h.s. of (2.3) is bounded by
$$
\begin{aligned}
&\mathbb E_{g_1, \ldots, g_N} \big[|\tau_{g_N\cdots g_1}(\theta)
-\psi'|<\ve\big]\qquad (\psi'=\psi^\bot)\\
&=\sum_g \mu^{(N)}(g) \ 1_{[\psi' -\ve, \psi'+\ve]} (\tau_g(\theta)).
\end{aligned}
\eqno{(2.4)}
$$
Let $0\leq f\leq 1$ be a smooth function on $S^1$ such that $f=1$ on
$[\psi'-\ve, \psi' +\ve]$, supp\,$f \subset [\psi' -2\ve, \psi'+2\ve]$
and $|\partial^\alpha f|\lesssim_\alpha \big(\frac 1\ve\big)^\alpha$.
Then, invoking the large deviation estimate for the $\mu$-random walk,
we obtain
$$
\begin{aligned}
(2.4)&\leq \sum_g \mu^{(N)} (g) (f\circ\tau_g)\\
&\leq \int f d\nu_E + e^{-c(\lambda)N}\Vert f\Vert_{C^1}\\
&\leq \nu_E ([\psi'-2\ve, \psi'+2\ve])+\frac 1\ve \, e^{-c(\lambda)N}
\end{aligned}
$$
proving the lemma.
\end{proof}

\begin{lemma}\label{Lemma2}
Assume the Lyapounov exponent H\"older regular of exponent $\alpha>0$.
Then
$$
\max_{|E-E_0|<\kappa} \log \Vert M_N(E)\Vert < L(E_0)N+c\kappa^\alpha N
\eqno{(2.5)}
$$
outside a set $\Omega$ of measure at most $e^{-\kappa'N}$.
\end{lemma}

\begin{proof}
Recall the large deviation theorem for the Lyapounov exponent
$$
\mathbb E\Big[\Big|\frac 1N\log \Vert M_N(E)\Vert - L(E)\Big|>\sigma
\Big] \lesssim e^{-\sigma'N}.
\eqno{(2.6)}
$$
Set $\kappa_1 =C\kappa^\alpha$.
It follows from (2.6) that for $|E-E_0|<\kappa$
$$
\mathbb E[\log \Vert M_N(E)\Vert> L(E_0)N+ \kappa_1 N]\lesssim
e^{-\kappa'N}\eqno{(2.7)}
$$
since $|L(E)-L(E_0)|\lesssim\kappa^\alpha$.

More generally, given indices $N>\ell_1>\ell_2> \cdots> \ell_r>1$
$\big(r=O(1)\big)$, we have for $|E-E_0|<\kappa$ that
$$
\begin{aligned}
\mathbb E[\log \Vert M_{N-{\ell_1}}^{v_N, \ldots, v_{\ell_1+1}}
(E)\Vert&+\log \Vert M_{\ell_1-\ell_2-1}^{(v_{\ell_1-1}, \ldots,
v_{\ell_2+1})}(E)\Vert +\cdots+\log\Vert M_{\ell_r-1}^{(v_{\ell_r-1},
\ldots, v_1)}(E)\Vert\\
&> L(E_0)N+\kappa_1 N] \lesssim
e^{-\kappa'N}.
\end{aligned}
\eqno{(2.8)}
$$

Set $\theta =e^{-\frac 12 \kappa' N}$ and let $\mathcal E$ be a finite subset of
$|E-E_0|<\kappa$, $|\mathcal E|<\frac 1\theta$ such that
$\max_{|E-E_0|<\kappa} \dist (E, \mathcal E)<\theta$.
Take $r=r(\kappa)$ and $\Omega\subset \{1, -1\}^N$ such that (2.8) holds
for $V\not\in\Omega, E\in\mathcal E$ and all $N>\ell_1>\cdots>
\ell_r>1$,
$$
|\Omega|\lesssim N^r.|\mathcal E| e^{-\kappa'N} < e^{-\frac
13\kappa'N}.\eqno{(2.9)}
$$
Take then $E\in [E_0-\kappa, E_0+\kappa]$ and $E_1\in\mathcal E|, E-E_1|<\theta$.
Using a truncated Taylor expansion, we get
$$
\begin{aligned}
&M_N(E)=M_N(E_1) +\Big[\sum_{1\leq \ell\leq N}
M_{N-\ell}(E_1)\begin{pmatrix} 1&0\\ 0&0\end{pmatrix} M_{\ell-1}
(E_1)\Big] (E-E_1)+\cdots\\
&+\frac 1{r!} \Big[\sum_{1\leq \ell_1< \ell_2<\cdots<\ell_r\leq N}
M_{N-\ell_1} (E_1) \begin{pmatrix} 1&0\\0&0\end{pmatrix}
M_{\ell_1-\ell_2-1} (E_1) \begin{pmatrix} 1&0\\0&0\end{pmatrix}\cdots
M_{\ell_r-1} (E_1)\Big] (E-E_1)^r\\[12pt]
&+O(C^N\theta^{r+1}).\end{aligned}
\eqno{(2.10)}
$$ 
Taking $r\sim \frac 1{\kappa'}$, we ensure the remainder term $<
e^{-N}$, while for $V\not\in \Omega$
$$
\Vert(1.8)\Vert < e^{(L(E_0)+\kappa_1)N} +\theta
e^{(L(E_0)+\kappa_1)N}+\cdots+ \theta^r e^{(L(E_0)+\kappa_1)N}
$$
Lemma 2 follows.
\end{proof}

\medskip
\section
{A Wegner estimate}

\begin{proposition}\label{Proposition3}
Assume the Furstenberg measures of $H$ have bounded density.
Then
$$
\mathbb E[\Spec H_N\cap I\not= \phi]< CN.|I|+C e^{-cN}\eqno{(3.1)}
$$
if $I\subset R$ is an interval.

\end{proposition}
\medskip

\begin{proof}
What follows is an adaptation of the argument used in \cite{B2}.
Let $I=[E_0-\delta, E_0+\delta]$ and assume $H_N$ has an eigenvalue
$E\in I$ with eigenvector $\xi =(\xi_j)_{1\leq j\leq N}$.
Then
$$
M_N(E)\begin{pmatrix} \xi_1\\ 0\end{pmatrix} =\begin{pmatrix} 0\\ \xi_N
\end{pmatrix}.
$$
Assume $|\xi_1|\geq |\xi_N|$ (otherwise, we replace $M_N$ by
$M_N^{-1}$).
It follows that
$$
\Vert M_N(E) e_1\Vert\leq 1.\eqno{(3.2)}
$$
On the other hand, from the large deviation theorem
$$
\log \Vert M_N(E_0)e_1\Vert > L(E_0)N-\kappa N\eqno{(3.3)} 
$$
for $V\not\in\Omega$, where
$$
|\Omega|< e^{-\kappa'N}.\eqno{(3.4)}
$$
Here $\kappa>0$ is an appropriate constant.

In view of Lemma 2, we may moreover assume that for $|E-E_0|<\delta$
$$
\max_n\Vert M_{N-n}^{(v_N, \ldots, v_{n+1})} (E)\Vert \
\Vert M_{n-1}^{(v_{n-1}, \ldots, v_n)} (E)\Vert < e^{L(E_0)N+\kappa
N}\eqno{(3.5)}
$$
if $V\not\in \Omega$.

Denote
$$
B=e^{L(E_0)N-2\kappa N}.
$$
Then, for $V\not\in \Omega$,
$$
\begin{aligned}
\kappa N&\leq \big|\log (\Vert M_N(E)e_1\Vert+B)-\log (\Vert
M_N(E_0)e_1\Vert +B)\big|\\
&\leq \int_{-\delta}^{\delta} \Big|\frac d{dt} \log (\Vert
M_N(E_0+t)e_1\Vert+B)\big| dt.
\end{aligned}\eqno{(3.6)}
$$
The integrand in (3.6) is clearly bounded by
$$
\sum_{j=1, 2}\, \sum^N_{n=1} \frac{|\langle M_{N-n}^{(v_N, \ldots,
v_{n+1})} (E_0+t) e_1, e_j\rangle|\, |\langle M_{n-1}^{(v_{n-1}, \ldots,
v_1)}(E_0+t)e_1, e_1\rangle |}{\Vert M_N^{(v_N, \ldots, v_1)}
(E_0+t) e_1\Vert +B}
\eqno{(3.7)}
$$
and we estimate the $n$-norm by
$$
\frac {\Vert \big(M_{N-n}^{(v_N, \ldots, v_{n+1})}(E_0+t)\big)^* 
e_j\Vert .\Vert M_{n-1}^{(v_{n-1}, \ldots, v_1)} (E_0+t) e_1\Vert}
{|\langle\big( M_{N-n}^{v_n, \ldots, v_{n+1}}(E_0+t)\big)^* e_j,
M_{n-1}^{(v_{n-1}, \ldots, v_1)} (E_0+t) e_1\rangle|}.\eqno{(3.8)}
$$
We distinguish two cases.
If $n\geq \frac N2$, set
$$
\eta =\frac {(M_{N-n}^{(v_N, \ldots, v_{n+1})} (E_0+t)\big)^* e_j}
{\Vert (M_{N-n}^{(v_{N},\ldots, v_{n+1})} (E_0+t)^* e_j\Vert}
$$
which is independent from $v_1, \ldots, v_n$.
From Lemma 1, we get the distributional inequality
$$
\mathbb E_{v_1, \ldots, v_{n-1}}[|\langle M_{n-1}^{v_{n-1}, \ldots,
v_1)}(E_0+t) e_1 , \eta\rangle |< \ve\Vert M_{n-1}^{(v_{n-1}, \ldots,
v_1)} (E_0+t)e_1\Vert]\leq C\ve\eqno{(3.9)}
$$
since by assumption $\tau(E) \leq C\ve$.
If $n<\frac N2$, set
$$
\eta =\frac {M_{n-1}^{v_{n-1}, \ldots, v_1)} (E_0+t) e_1}
{\Vert M_{n-1}^{(v_{n-1}, \ldots, v_1)}(E_0+t)e_1\Vert}
$$
and argue similarly, considering $\big( M_{N-n}^{(v_N, \ldots, v_{n-1})} (E_0+t)\big)^* $.

Hence we proved that
$$
\mathbb E_{v_1, \ldots, v_N}[(3.8)>\lambda]< C\lambda^{-1}\text { for }
\lambda <c\log N.\eqno{(3.10)}
$$
On the other hand, the $n$-term in (3.7) is also bounded by
$$
\frac 1B\Vert M_{N-n}^{(v_N, \ldots, v_{n+1})}(E_0+t)\Vert.\Vert
M_{n-1}^{(v_{n-1},\ldots, v_1)} (E_0+t)\Vert < e^{3\kappa
N}\eqno{(3.11)}
$$
since $V\not\in \Omega$, by (3.5).
Therefore, taking $\kappa$ in (3.11) appropriately, according to (3.10),
it follows that
$$
\mathbb E[(3.7)\, 1_{\Omega^c}] \lesssim N^2.\eqno{(3.12)}
$$
Consequently, recalling (3.6), we obtain from (3.12) and Tchebychev's
inequality
$$
\mathbb E[\Spec H_N\cap I\not= \phi]\leq |\Omega|+C\frac \delta{\kappa
N}
N^2< e^{-\kappa'N}+\frac c\kappa \delta N
$$
proving (3.1).
\end{proof}
\medskip

Let $H$ be as in Proposition 1 on the sequel

The energy range is restricted to $[-2+\delta_0, 2-\delta_0]$ according
to Proposition~1.

Using Proposition 3 and Anderson localization, one deduces then the
analogue of Proposition 3 in \cite{B2}.
We leave the details to the reader (see \cite{B2}).

\begin{proposition}\label{Proposition4}

Assuming $\log \frac 1\delta < c(\lambda) N$, we have for
$I=[E_0-\delta, E_0+\delta]$ that
$$
\mathbb E[Tr \mathcal X_I(H_N)] =Nk(E_0)|I| +O\Big(N\delta^2+\delta \log^2\Big(N+\frac 1\delta\Big)\Big).
\eqno{(3.13)}
$$
\end{proposition}

\section
{Near resonances}

In what follows, we develop an alternative to Minami's argument that is applicable in the A-B context
(recall that $H=H_\lambda$ with $\lambda$ satisfying the conditions from Proposition 1).

\begin{lemma}\label{Lemma3}
Let $I={E_0-\delta, E_0+\delta}$ be as above. Let $N\in\mathbb Z_+$.

The probability for existence of a pair of orthogonal unit vectors $\xi, \xi'\in\mathbb R^n$ satisfying
$$
\Vert (H_N-E_0)\xi\Vert_2 <\delta, \Vert(H_N-E_0)\xi'\Vert_2<\delta\eqno{(4.1)}
$$
$$
\max_{j<\sqrt N} (|\xi_j|, |\xi_{N-j}|, |\xi_j'|, |\xi_{N-j}'|)<\frac 1{N^{10}}\eqno{(4.2)}
$$
is at most
$$
CN^7 \delta^2 +e^{-c\sqrt N}.\eqno{(4.3)}
$$
\end{lemma}

\begin{proof}

We take $\sqrt N<\nu<N-\sqrt N$ such that $|\xi_\nu|\gtrsim \frac 1{\sqrt N}$.
Since $\xi \bot\xi', \hfill\break
\Vert\xi_\nu\xi'-\xi_\nu'\xi\Vert_2 \gtrsim \frac 1{\sqrt N}$ and there is some $\nu_1$ so that
$$
|\xi_\nu \xi_{\nu_1}'-\xi_\nu' \xi_{\nu_1}|\frac 1N.\eqno{(4.4)}
$$
Again by (4.2). $\sqrt N<\nu_1< N-\sqrt N$.
Set further for $1\leq j\leq N$
$$
(\xi_j, \xi_j')=\big(\xi_j^2 +(\xi_j')^2\big)^{\frac 12} e^{i\theta_j}.\eqno{(4.5)}
$$
Hence (4.4) certainly implies that
$$
|\sin(\theta_\nu -\theta_{\nu_1})|>\frac 1N.
\eqno{(4.6)}
$$
Assume $\nu<\nu_1$ (the other alternative is similar).
We distinguish two cases.
\medskip

\noindent
{\bf Case 1.} {\sl There is some $\nu<j_1<\nu_1$ such that
$$
|\sin(\theta_\nu -\theta_{j_1})|>\frac 1{10N} \text { and } |\sin(\theta_{\nu_1} -\theta_{j_1})|>\frac
1{10N}.\eqno{(4.7)}
$$}

Define the vector
$$
\eta=\frac {\xi_{j_1}\xi'-\xi_{j_1}'\xi}{(\xi^2_{j_1}+(\xi_{j_1}')^2)^{\frac 12}}.
$$
Obviously $\Vert\eta\Vert_2=1$ and $\eta_{j_1}=0$.
Also, from (4.1) it easily follows that
$$
\Vert(H_N-E_0)\eta\Vert_2 < 2\delta.\eqno{(4.8)}
$$
From (4.7)
$$
|\eta_\nu|= \big(\xi^2_\nu+(\xi_\nu')^2\big)^{\frac 12} |\sin (\theta_{j_1}-\theta_\nu)|\gtrsim \frac 1{N^{3/2}}\eqno{(4.9)}
$$
$$
|\eta_{\nu_1}|\gtrsim \frac 1{N^2}.\eqno{(4.10)}
$$
Next, we introduce the vectors
$$
\eta^{(1)}=\eta|_{[1, j_1[} \text { and } \eta^{(2)} =\eta|_{[j_1+1, N]}
$$
as well as the restrictions
$$
H^{(1)} = H_{[1, j_1[} \text { and } H^{(2)} = H_{[j_1+1, N]}\eqno{(4.11)}
$$
with Dirichlet boundary conditions.
By (4.9), (4.10), $\Vert\eta^{(1)}\Vert_2,  \Vert \eta^{(2)}\Vert_2\gtrsim \frac 1{N^2}$ while by (4.8) and $\eta_{j_1} =0$,
it follows that $\Vert (H^{(1)}-E_0)\eta^{(1)} \Vert_2<2\delta, \hfill\break
\Vert (H^{(2)}-E_0)\eta^{(2)}\Vert_2< 2\delta$.

Hence
$$
\dist (E, \Spec H^{(1)})<N^2\delta, \dist (E, \Spec H^{(2)})<N^2\delta.\eqno{(4.12)}
$$
Note that $H^{(1)}$, $H^{(2)}$ are independent as functions of $V=(v_n)_{1\leq n\leq N}$ and by construction,
$\sqrt N\leq j_1\leq N-\sqrt N$.
Involving Proposition 3, it follows that the probability for the joint event (4.12) is at most
$$
c[j_1N^2\delta+ e^{cj_1}][(N-j_1)N^2\delta+ e^{-c(N-j_1)}]< CN^6\delta^2+ e^{-c\sqrt N}.\eqno{(4.13)}
$$
\medskip

\noindent
{\bf Case 2.} {\sl For all $\nu\leq j\leq \nu_1$, either $|\sin(\theta_\nu -\theta_j)|\leq \frac 1{10N} $ or
$|\sin(\theta_{\nu_1}-\theta_j)|\leq \frac 1{10N}$.}

Take then the smallest $\nu< j_1\leq \nu_1$ for which $|\sin(\theta_{\nu_1}-\theta_{j_1})|\leq \frac 1{10N}$.
Hence $|\sin(\theta_\nu -\theta_{j_1-1})|\leq \frac 1{10N}$. Denote
$$
\begin{aligned}
\eta^{(1)}&= \frac {\xi_{j_1}\xi'-\xi_{j_1}'\xi}{(\xi^2_{j_1} +(\xi_{j_1}')^2)^{\frac 12}} \Big|_{1\leq j\leq j_1-1}\\
\eta^{(2)}&= \frac {\xi_{j_1-1} \xi' -\xi'_{j_1-1}\xi}{(\xi^2_{j_1-1}+(\xi_{j_1-1}')^2)^{\frac 12}}\Big|_{j_1\leq j\leq N}
\end{aligned}
$$
and $H^{(1)} =H_{[1, j_1[}, H^{(2)} = H_{[j_1, N]}$.
Since $\eta^{(1)}_{j_1} =0,  \Vert(H^{(1)}-E_0)\eta^{(1)}\Vert_2< 2\delta$.
Also $\Vert\eta^{(1)}\Vert_2\geq |\eta_\nu|\gtrsim \frac 1{\sqrt N}|\sin(\theta_\nu-\theta_{j_1})|\gtrsim
\frac 1{\sqrt N}\big(\frac 1N-\frac 1{10N}\big)>\frac 1{N^2}$, implying $\dist (E, \Spec H^{(1)})<N^2\delta$.
Similarly $\dist (E, \Spec H^{(2)})< N^2\delta$ and we conclude as in Case 1.

Summing (4.13) over $j$, Lemma \ref{Lemma3} follows.
\end{proof}

We may now establish an analogue of Proposition 4 in \cite{B2} for Anderson-Bernoulli Hamiltonians as considered above.

\begin{proposition}
\label{Proposition5}
Let $I=[E_0-\delta, E_0+\delta]\subset [-2+\delta_0], [2+\delta_0] \text { and }  \log \frac 1\delta <c\sqrt N$.

Then
$$
\mathbb E[H_N \text { has at least two eigenvalues in $I$}]\leq CN^2\delta^2+C\delta \log \Big(N+\frac
1\delta\Big).\eqno{(4.14)}
$$
\end{proposition}

\medskip
\begin{proof}

Proceeding as in \cite{B2}, set $M=C\log^2\big(N+\frac 1\delta\big)$ for an appropriate constant $C$.
From the theory of Anderson localization, the eigenvectors $\xi_\alpha$ of $H_N$ $|\xi_\alpha|=1$, satisfy
$$
|\xi_\alpha(j)| < e^{-c|j-j_\alpha|} \text { for } |j-j_\alpha|>\frac M{10}\eqno{(4.15)}
$$
with probability at least $1-e^{-cM}$, with $j_\alpha$ the center of localization of $\xi_\alpha$.
We may therefore introduce a collection $(\Lambda_s)_{1\leq s\leq s_1}, s_1\lesssim \frac NM$, of size $M$ subinterval of
$[1, N]$ such that for each $\alpha$, there is some $1\leq s\leq s_1$ satisfying
$$
j_\alpha \in\Lambda_s \text { and } \Vert\xi_\alpha|_{[1, N]\backslash\Lambda_s} \Vert_2 < e^{-cM}\eqno{(4.16)}
$$
$$
\Vert (H_{\Lambda_s}-E_\alpha)\xi_{\alpha, s}\Vert_2 < e^{-cM}.\eqno{(4.17)}
$$
where $\xi_{\alpha, s} =\xi_\alpha\big|_{\Lambda_s}$.
For $1<s<s_1$, we may moreover insure that
$$
|\xi_\alpha(j)|< e^{-cM} \text { if } \dist (j, \partial\Lambda_s)<\frac M{10}.\eqno{(4.18)}
$$
By (4.16), (4.17), $\dist(E_\alpha, \Spec H_{\Lambda_s})<e^{-cM}$ and hence $\Spec H_{\Lambda_s}\cap \tilde I\not=\phi$,
$\tilde I=[E_0-2\delta, E_0+2\delta]$, if $E_\alpha\in I$.
According to Proposition 3, by our choice of $M$
$$
\mathbb E[\Spec H_{\Lambda_s}\cap \tilde I\not= \phi]< CM\delta+ce^{-cM}<CM\delta. \eqno{(4.19)}
$$
Note that if $\Lambda_s\cap \Lambda_{s'}=\phi$, then $H_{\Lambda_s}, H_{\Lambda_{s'}}$ are independent.

Hence, by construction and (4.19),
$$
\begin{aligned}
&\mathbb E\, [\text{there are $\alpha, \alpha'$ s.t. $E_\alpha, E_{\alpha'}\in I$ and $ |j_\alpha -j_{\alpha'}|> 4M$}]\leq\\
&C\sum_{s, s'} (M\delta)^2 <CN^2\delta^2.
\end{aligned}\eqno{(4.20)}
$$
with the $s, s'$-sum performed over pairs such that $\Lambda_s \cap \Lambda_{s'}=\phi$.

It remains to consider the case $|j_\alpha -\j_{\alpha'}|\leq 4M$.
If $\dist(j_\alpha, \{1, N\})< 2M$ then $\Spec H_{[1, 4M]}\cap \tilde I \not= \phi$ or $\Spec H_{[N-4M, N]}\cap \tilde I
\not=\phi$.

Again by Proposition 3, the probability for this event is less than
$$
CM\delta <C\log \Big(N+\frac 1\delta\Big)\delta.\eqno{(4.21)}
$$
Next assume moreover $\dist(\{j_\alpha, j_{\alpha'}\}, \{1, N\})\geq 2M$. Then
$$
j_\alpha \in\Lambda_t, j_{\alpha'}\in \Lambda_{t'} \text{ where  $1<t, t'<s_1$ and $|t-t'|<10$}.
$$

Introduce an interval $\Lambda$ obtained as union of at most 10 consecutive $\Lambda_s$ intervals, such that
$\Lambda_t, \Lambda_{t'}\subset\Lambda$.
By (4.16), (4.18), setting $\tilde \xi_\alpha =\xi_\alpha|_\Lambda, \tilde\xi_{\alpha'} =\xi_{\alpha'}|_\Lambda$, we get
$$
\Vert (H_\Lambda -E_\alpha)\tilde\xi_\alpha\Vert_2 < e^{-cM}, \Vert (H_\Lambda -E_{\alpha'}) \tilde\xi_{\alpha'}\Vert_2
< e^{-cM}
$$
so that for $E_\alpha, E_{\alpha'}\in I$
$$
\Vert(H_\Lambda -E_0)\tilde\xi_\alpha\Vert_2<2\delta, \Vert(H_\Lambda -E_0)\tilde\xi_{\alpha'}\Vert_2< 2\delta.
\eqno{(4.22)}
$$
Also, by (4.18), $\max_{\dist (j, \partial\Lambda)<\frac M{10}} (|\xi_\alpha (j)|, |\xi_{\alpha'}(j)|)<
e^{-cM}<\frac 1{|\Lambda|^{10}}$.
Hence, Lemma 3 applies to $H_\Lambda$.
According to (4,3), the probability that $H_\Lambda$ satisfies the above property is at most (again by our choice of $M$)
$$
CM^7 \delta^2+e^{-c\sqrt M}<CM^7\delta^2.\eqno{(4.23)}
$$
Summing over the different boxes $\Lambda$ introduced above gives then
$$
CN.M^7\delta^2< CN\Big(\log \Big(N+\frac 1\delta\Big)\Big)^7\delta^2.\eqno{4.24)}
$$
Adding the contributions (4.20), (4.21), (4.24) and noting that the last is majorized by the first two, inequality
(4.14) follows.
\end{proof}

\section
{Local eigenvalue statistics}

Following the same argument as in Proposition 5 of \cite {B2}, Proposition 4 and Proposition 5 above permit to establish
Poisson statistics for the local eigenvalue spacings.  Thus we obtain Proposition 2 stated in the Introduction.

The proof is completely analogous to that of Proposition 5 in \cite{B2}, except that instead of choosing $M=K\log N$,
$M_1=K_1\log N$, we take say $M=(\log N)^4, M_1 =(\log N)^3$.

\medskip

\noindent
{\bf Acknowledgment.} The author is grateful to the UC Berkeley mathematics department (where the paper was written) for its
hospitality.

\end{document}